\documentclass[a4paper]{article}
\usepackage{amsthm,amssymb,amsmath,graphicx,psfrag,enumitem,mathrsfs}
\usepackage{algorithm}
\usepackage{tikz}
\usepackage{contour}

\usetikzlibrary{backgrounds,patterns}

\newtheorem{definition}{Definition}

\newtheorem{theorem}[definition]{Theorem}

\newtheorem{lemma}[definition]{Lemma}

\newcommand{\comment}[1]{}

\newcommand{\cP}{\mathcal{P}}
\newcommand{\cQ}{\mathcal{Q}}

\newcommand{\cA}{\mathcal{A}}

\newcommand{\cL}{\mathcal{L}}
\newcommand{\cR}{\mathcal{R}}

\newcommand{\EPP}{Erd\H{o}s-P\'osa property}

\newcommand{\sm}{\setminus}

\tikzstyle{vx}=[thick,circle,inner sep=0.cm, minimum size=2.mm, fill=white, draw=black]
\tikzstyle{svx}=[vx,fill=dunkelgrau]
\tikzstyle{tvx}=[vx,fill=black]
\tikzstyle{edg}=[draw,very thick]
\colorlet{hellgrau}{black!20!white}
\colorlet{dunkelgrau}{black!50!white}
\colorlet{hellblau}{blue!20!white}
\colorlet{hellrot}{red!40!white}

\title{Connecting vertex sets to walls}
\author{Henning Bruhn\thanks{The research leading to these results was partially supported by the Deutsche Forschungsgemeinschaft (DFG, German Research Foundation) -- 321904558} \and Felix Joos\thanks{The research leading to these results was partially supported by the Deutsche Forschungsgemeinschaft (DFG, German Research Foundation) -- 428212407.} \and Arthur Ulmer\footnotemark[1]}

\date{}

\begin{document}

\maketitle


\begin{abstract}
Menger's theorem on $A$--$B$-paths and Gallai's theorem on $A$-paths are among the most useful results in structural graph theory.
Many variants and extensions are known.
We add to this line of research and prove results that relate the maximal number of vertex-disjoint paths between vertex sets and a wall to the minimum number of vertices meeting all these paths.
We also include types of paths that start and end in the wall.
\end{abstract}

\section{Introduction}

Menger's theorem is one of the most fundamental theorems in graph theory:
either a graph contains $k$ disjoint paths connecting two vertex sets $A$ and $B$,
or there is a set of at most $k-1$ vertices that separates $A$ from $B$.
The theorem is applied in many proofs, often as a tool to piece together some larger 
structure. \emph{Walls} (see Figure~\ref{fig:elwall}, and the next section for 
formal definitions)
are often used in a similar role, as building blocks,
and often Menger's theorem and walls are used in conjunction.

What then happens in Menger's theorem if $B$ is a wall $W$?
In one sense, nothing special: unless a small vertex set separates $A$ from $W$,
there is a large set of disjoint paths linking $A$ to $W$. These paths, however,
can end wildly in~$W$, perhaps all in the same subdivided wall edge. 
If we use Menger's theorem and walls as building blocks, we typically would like
to have more control over the paths, for instance, that they attach in an orderly fashion 
to $W$, and only to its outside (to the nails, perhaps). 

A moment's thought makes it clear that this is an impossible demand, as the set $A$ could
already lie in the wall. If we allow $W$ to be replaced by a subwall, however, we obtain
a true statement, a sort of wall version of Menger's theorem:
\begin{theorem}\label{wallmengerthm}
	For all positive integers $s',t'$, 
	there are integers $t$ and $f$ such that the following holds.
	If $A$ is a vertex set, and~$W$ is a wall of size at least~$t$, 
	then either
	\begin{enumerate}[label=\rm(\roman*)]
		\item there is a vertex set $X$ of size at most~$f$ 
		that separates $A$ from the set of branch vertices of $W$; or
		\item there is a subwall $W'$ of size at least~$t'$ and a set of $s'$ disjoint $A$--$W'$-paths
such that each path ends in a different nail of $W'$. 
	\end{enumerate}
\end{theorem}
(An $A$--$W'$-path has an endvertex in $A$, an endvertex in $W'$, and is internally
disjoint from $A\cup W'$.)

Observe that we cannot replace the branch vertices of $W$ by the vertex set of $W$ in~(i).
A set $A$ may be very well connected to a path between two branch vertices $b,b'$ of $W$,
and hence $A$ and the vertex set of $W$ cannot be separated by a set of size at most~$f$.
However, it may be possible that the branch vertices $b,b'$ already separate $A$ from all 
other branch vertices of $W$. Then there cannot be a large set of paths as in~(ii).

We find it useful to generalise Theorem~\ref{wallmengerthm}, so that not one but many 
vertex sets can be connected to the wall simultaneously. 
For this, 
let $\mathcal A$ be a set of $n$ vertex sets, 
and let $W$ be a wall with a set $N$ of nails.
Then a set $\mathcal P$ of disjoint paths is an \emph{$\mathcal A$--$W$-connector} of size $s$ 
if $\mathcal P$ can be partitioned into $n$ sets $\mathcal P_A$, $A\in\mathcal A$, 
such that $\mathcal P_A$ consists of at least $s$ (disjoint) $A$--$N$-paths that are disjoint from $W-N$;
see Figure~\ref{connectorfig} for an illustration.

Our first main result is:
\begin{theorem}\label{connlem}
	For all positive integers $n,s',t'$, 
	there are integers $t=t(n,s',t')$ and $f=f(n,s')$ such that the following holds.
	If $\mathcal A$ is a set of $n$ vertex sets, and~$W$ is a wall of size at least~$t$ with $B$ as its set of branch vertices, 
	then either
	\begin{enumerate}[label=\rm(\roman*)]
		\item there is $A\in\mathcal A$ and a vertex set $X$ of size at most~$f$ 
		such that $X$ separates $A$ from $B$; or
		\item there is a subwall $W'$ of size at least~$t'$ and there 
		is an $\mathcal A$--$W'$-connector of size at least~$s'$.
	\end{enumerate}
\end{theorem}

\begin{figure}[htb]
\centering
\begin{tikzpicture}[scale=0.6]

\def\height{2}

\newcommand{\nail}[1]{
\fill[dunkelgrau,rounded corners=1pt] (#1-0.35,\height-0.1) rectangle ++(0.7,0.2);
}

\pattern[pattern=bricks,pattern color=dunkelgrau] (0,0) rectangle (7,\height);
\draw[ultra thick, black, rounded corners=1pt] (0,0) to ++(0,\height) to ++(7,0) to ++(0,-\height);

\begin{scope}[opacity=0.5]
\draw[line width=0.5cm,dunkelgrau,out=-90,in=90] (1.5,2*\height) to (1,\height); 
\draw[line width=0.5cm,dunkelgrau,out=-90,in=90] (3.5,2*\height) to (3,\height); 
\draw[line width=0.5cm,dunkelgrau,out=-90,in=90] (5.5,2*\height) to (6,\height); 
\end{scope}

\fill[dunkelgrau,rounded corners=2pt] (1,2*\height) rectangle ++(1,0.4);
\fill[dunkelgrau,rounded corners=2pt] (3,2*\height) rectangle ++(1,0.4);
\fill[dunkelgrau,rounded corners=2pt] (5,2*\height) rectangle ++(1,0.4);

\nail{1}
\nail{3}
\nail{6}

\node at (1.2,\height+1) {$\mathcal P_2$};
\node at (3.2,\height+1) {$\mathcal P_1$};
\node at (5.8,\height+1) {$\mathcal P_3$};

\node at (0.6,2*\height+0.) {$A_2$};
\node at (2.6,2*\height+0.) {$A_1$};
\node at (6.4,2*\height+0.) {$A_3$};

\begin{scope}[shift={(9,0)}]
\pattern[pattern=bricks,pattern color=dunkelgrau] (0,0) rectangle (7,\height);
\draw[ultra thick, black, rounded corners=1pt] (0,0) to ++(0,\height) to ++(7,0) to ++(0,-\height);

\begin{scope}[opacity=0.5]
\draw[line width=0.5cm,dunkelgrau,out=90,in=90] (1,\height) arc [start angle=180, end angle=0, radius=1]; 
\draw[line width=0.5cm,dunkelgrau,out=90,in=90] (2,\height) arc [start angle=180, end angle=0, radius=2]; 
\draw[line width=0.5cm,dunkelgrau,out=90,in=90] (4,\height) to ++(0,0.3) arc [start angle=180, end angle=0, radius=0.5]
to ++(0,-0.3); 
\end{scope}

\fill[dunkelgrau,rounded corners=2pt] (3.8,\height+1.5) rectangle ++(0.4,1);
\fill[dunkelgrau,rounded corners=2pt] (4.3,\height+0.3) rectangle ++(0.4,1);
\fill[dunkelgrau,rounded corners=2pt,rotate=30] (2.5,\height-0.75) rectangle ++(0.4,1);

\nail{1}
\nail{2}
\nail{3}
\nail{4}
\nail{5}
\nail{6}

\node at (1.1,\height+0.5) {$\mathcal L_1$};
\node at (5.1,\height+0.5) {$\mathcal L_2$};
\node at (5.8,\height+1) {$\mathcal L_3$};

\node at (0.6,\height+1.3) {$A_1$};
\node at (3.4,\height+2.3) {$A_3$};
\node at (3.8,\height+1.) {$A_2$};

\end{scope}
\end{tikzpicture}
\caption{Left: An $\mathcal A$--$W$-connector. Right: An $\mathcal A$-linkage of $W$.
Note that the endvertices separate nicely. While this is not immediately required for connectors
and linkages, we will see later that this can easily be guaranteed.}\label{connectorfig}
\end{figure}
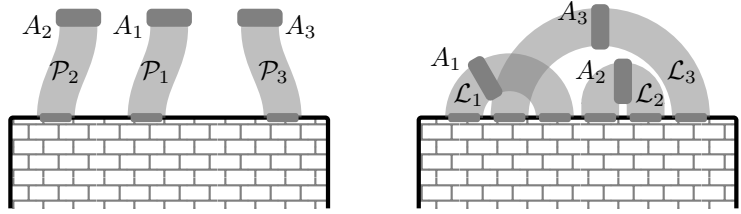

Gallai's theorem can be viewed as an analogue of Menger's theorem for \emph{$B$-paths},
non-trivial paths with both ends in $B$ that are internally disjoint from~$B$.
Gallai's theorem states that
whenever there do not exist $k$ disjoint $B$-paths in a graph $G$,
then there is a set of at most $2k-2$ vertices that intersects all $B$-paths.
Kakimura, Kawarabayashi and Marx~\cite{KKM11} generalised Gallai's theorem to 
\emph{$B$--$A$--$B$-paths}, $B$-paths that each meet a second vertex set~$A$ (possibly in
an endvertex).
\begin{theorem}[Kakimura, Kawarabayashi and Marx~\cite{KKM11}]\label{STSpath}
Let $A,B$ be two vertex sets\footnote{To be precise, Kakimura et al.\ only state this theorem for the case when $A,B$ are disjoint.
However, this condition can be easily removed;
given an instance where $A,B$ are not necessarily disjoint, construct a graph $H$ which arises from $G$ by subdividing every edge incident to a vertex in~$A$,
and let~$A'$ be the set of all these new vertices. 
Thus, $A',B$ are disjoint and an application of the theorem for disjoint vertex sets to $H$ and $A',B$ proves the result.} 
in a graph $G$, and let $k$ be a positive integer.
Then there are either $k$ disjoint $B$--$A$--$B$-paths in $G$ or a set of at most $2k-2$ vertices that intersects 
all $B$--$A$--$B$-paths.
\end{theorem}

We may now ask the same question as above: what happens if $B$ is a wall? Can we force the $B$--$A$--$B$-paths
to attach only to the nails of the wall?

Let $\cA$ be a set of~$n$ vertex sets, let $W$ be a wall, and let $N$ be a set of nails of~$W$.
A \emph{linkage} $\cL$ of~$W$ is a set of disjoint $N$-paths that are internally disjoint from $W-N$.
A linkage~$\mathcal L$ of a wall $W$ is an \emph{$\mathcal A$-linkage of $W$} of size at least $s$ 
if~$\cL$ can be partitioned into $n$ linkages $\cL_A$, $A\in\mathcal A$, of $W$ 
such that for each~$A\in\mathcal A$,
the linkage~$\mathcal L_A$ consists of at least~$s$ paths, each of which is an $N$--$A$--$N$-path.
Our second main result is about $\cA$-linkages of walls, and should be viewed as a wall 
version of the theorem by Kakimura, Kawarabayashi and Marx.


\begin{theorem}\label{linkagelem}
For all positive integers $n,s',t'$, 
there are integers $t=t(n,s',t')$ and $f=f(n,s')$ such that the following holds.
If $\mathcal A$ is a set of $n$ vertex sets, and if $W$ is a wall of size at least~$t$
with $B$ as its set of branch vertices, then either
\begin{enumerate}[label=\rm(\roman*)]
\item there is $A\in\mathcal A$ and a vertex set $X$ of size at most~$f$ 
such that $X$ meets every $B$--$A$--$B$-path; or
\item there is a subwall $W'$ of size at least~$t'$ and there 
is an $\mathcal A$-linkage of $W'$ of size at least~$s'$.
\end{enumerate}
\end{theorem}

In fact, we can also combine Theorems~\ref{connlem} and~\ref{linkagelem} into a single result.
As this is slightly more technical, we defer its statement (Theorem~\ref{labeltheorem})
to the end of the article.

We apply Theorem~\ref{labeltheorem} in another article~\cite{BJU:21} concerning the \EPP{} of labelled graphs that contain a fixed labelled graph as a minor.

In Section 2, we introduce all necessary definitions. In Section 3, we provide a way to connect vertex sets to a wall or to disconnect one of these sets from that wall. 
In the last two sections, given such a set of paths, we clean these paths up, that is, make them disjoint and then sort the endvertices. With these results, we then easily 
prove Theorem~\ref{connlem} and Theorem~\ref{linkagelem}. Finally, as stated, we combine these two theorems into a single, more general result and prove it.

\section{Walls and linkages}

We follow Diestel~\cite{diestelBook17} for general graph-theoretic definitions.
In particular, given two vertex sets $A$, $B$, an \emph{$A$--$B$-path} is a path 
with an endvertex in $A$ and the other in $B$, and is internally
disjoint from $A$ and $B$. A vertex in $A\cap B$
counts as a (trivial) $A$--$B$-path.
In this section, we list some more notions that we need throughout the article.

An \emph{elementary wall $W$ of size $n$} arises
from the following graph by removing the vertices of degree~$1$:
take $\{v_{i,j}:i\in [n+1], j\in [2n+2]\}$ as the vertex set and 
consider all $v_{i,j}v_{k,\ell}$ to be edges, where $i\leq k$ and either $i=k$ and $|j-\ell| =1$, or $i=k-1$, $j=\ell$, and $j$ has the same parity as $i$;
see Figure~\ref{fig:elwall}.
A \emph{brick} is a cycle of length $6$.
The \emph{$i^\text{th}$ row} of an elementary wall is the induced subgraph on $v_{i,1},\ldots, v_{i,2n+2}$ for $i\in [n+1]$ (ignore the vertices that have been removed); 
this is a path. 
There is exactly one set of $n+1$ disjoint paths between the first row and the $(n+1)^\text{th}$ row; these paths are the \emph{columns} of an elementary wall. Starting at $v_{1,1}$, we can order the endvertices of the columns on the first row. The \emph{$j^\text{th}$ column} is the $j^\text{th}$ path in this order. 
We call the first row the \emph{top row} and the \emph{nails} are the vertices of degree $2$ in the top row that are adjacent to a vertex of degree $3$. 

\begin{figure}[htb]
	\centering
	\begin{tikzpicture}
		\tikzstyle{vx}=[thick,circle,inner sep=0.cm, minimum size=1.6mm, fill=white, draw=black]
		\tikzstyle{marked}=[line width=3pt,color=dunkelgrau]
		\tikzstyle{point}=[thin,->,shorten >=2pt,color=dunkelgrau]
		
		\def\wallheight{8}
		\def\brickheight{0.4}
		
		\pgfmathtruncatemacro{\lastrow}{\wallheight}
		\pgfmathtruncatemacro{\penultimaterow}{\wallheight-1}
		\pgfmathtruncatemacro{\lastrowshift}{mod(\wallheight,2)}
		\pgfmathtruncatemacro{\lastx}{2*\wallheight+1}

		\draw[edg] (\brickheight,0) -- (2*\wallheight*\brickheight+\brickheight,0);
		\foreach \i in {1,...,\penultimaterow}{
			\draw[edg] (0,\i*\brickheight) -- (2*\wallheight*\brickheight+\brickheight,\i*\brickheight);
		}
		\draw[edg] (\lastrowshift*\brickheight,\lastrow*\brickheight) to ++(2*\wallheight*\brickheight,0);
		
		\foreach \j in {0,2,...,\penultimaterow}{
			\foreach \i in {0,...,\wallheight}{
				\draw[edg] (2*\i*\brickheight+\brickheight,\j*\brickheight) to ++(0,\brickheight);
			}
		}
		\foreach \j in {1,3,...,\penultimaterow}{
			\foreach \i in {0,...,\wallheight}{
				\draw[edg] (2*\i*\brickheight,\j*\brickheight) to ++(0,\brickheight);
			}
		}

		\def\colind{5}
		\foreach \j in {2,4,6}{
			\draw[marked] (\colind*\brickheight,\j*\brickheight-2*\brickheight) -- ++ (0,\brickheight) -- ++(-\brickheight,0) -- ++(0,\brickheight) -- ++(\brickheight,0);
		}
		\draw[marked] (\colind*\brickheight,6*\brickheight) -- ++ (0,\brickheight) -- ++(-\brickheight,0) -- ++(0,\brickheight);
		
		\def\rowind{5}
		\foreach \i in {1,...,\lastx}{
			\draw[marked] (\i*\brickheight-\brickheight,\rowind*\brickheight) -- ++(\brickheight,0);
		}
		
		\draw[marked] (2*\wallheight*\brickheight,1*\brickheight) -- ++(0,\brickheight) coordinate[midway] (brx)
		-- ++(-2*\brickheight,0)
		-- ++(0,-\brickheight) -- ++(2*\brickheight,0);

		\foreach \i in {1,...,\lastx}{
			\node[vx] (w\i w0) at (\i*\brickheight,0){};
		}
		\foreach \j in {1,...,\penultimaterow}{
			\foreach \i in {0,...,\lastx}{
				\node[vx] (w\i w\j) at (\i*\brickheight,\j*\brickheight){};
			}
		}
		\foreach \i in {1,...,\lastx}{
			\node[vx] (w\i w\lastrow) at (\i*\brickheight+\lastrowshift*\brickheight-\brickheight,\lastrow*\brickheight){};
		}
		
		\foreach \i in {2,4,...,\lastx}{
			\node[vx,fill=white] (w\i w\lastrow) at (\i*\brickheight+\lastrowshift*\brickheight-\brickheight,\lastrow*\brickheight){};
		}
		
		\node[anchor=mid] (tr) at (\lastx*\brickheight+0.5,\wallheight*\brickheight+0.8){top row};
		\draw[point,out=270,in=0] (tr) to (w\lastx w\wallheight);
		
		\node[anchor=mid] (nails) at (\lastx*\brickheight-1,\wallheight*\brickheight+0.8){nails};
		\draw[point,out=180,in=90] (nails) to (w10w\wallheight);
		\draw[point,out=190,in=90] (nails) to (w12w\wallheight);
		\draw[point,out=200,in=90] (nails) to (w14w\wallheight);
\foreach \i in {2,4,...,16}{
\node[vx,fill=black] at (w\i w\wallheight){};
}
		
		\node[anchor=mid] (vp) at (0,\wallheight*\brickheight+0.8){column};
		\draw[point,out=0,in=90] (vp) to (w\colind w\wallheight);
		
		\node[align=center] (hp) at (\lastx*\brickheight+1.2,\rowind*\brickheight+0.8){row};
		\draw[point,out=270,in=0] (hp) to (w\lastx w\rowind);
		
		\node[align=center] (br) at (\lastx*\brickheight+1.2,1*\brickheight+0.8){brick};
		\draw[point,out=270,in=0] (br) to (brx);

\begin{scope}[on background layer]
\fill[hellgrau] (w6w1.center) -- (w12w1.center) -- (w12w2.center) -- (w13w2.center) 
--(w13w3.center) --(w12w3.center) --(w12w4.center) --(w6w4.center) 
--(w6w3.center) --(w7w3.center) --(w7w2.center) --(w6w2.center) -- cycle;
\end{scope}
		

\node[inner sep=1pt] (sub) at (3.25,-0.5) {subwall};
\coordinate (subsub) at (3.45,0.6);
\draw[very thick,shorten >=2pt,color=white,out=75,in=-110] (sub) to (subsub);
\draw[point,out=75,in=-110] (sub) to (subsub);

	\end{tikzpicture}
	\caption{An elementary $8$-wall}
	\label{fig:elwall}
\end{figure}
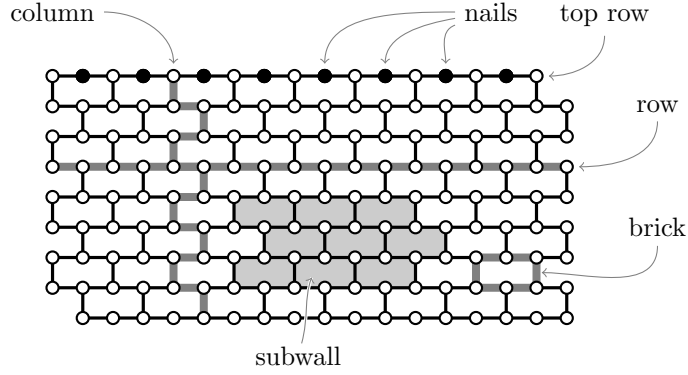

A \emph{wall of size $n$}, or \emph{$n$-wall}, 
is a subdivision of an elementary wall of size~$n$. Except for nails, all the definitions for elementary walls can be extended to walls in a natural way. In some special cases, the nails can be chosen in a natural way, too. We will explain that in a moment. The \emph{branch vertices} of a wall are the vertices of degree $3$.

We say that a wall $W_0$ is a \emph{subwall} of another wall $W$ if $W_0\subseteq W$ and there is a set of consecutive rows and columns of $W$ that give rise to $W_0$.
The subwall~$W_0$ is \emph{$t$-contained} if it is disjoint from the first and last $t$ rows and columns of~$W$. 

In the top row of a wall $W_0$, there is always at least one vertex of degree~$2$ in every brick of $W_0$ that could be a nail of the underlying elementary wall that gives rise to $W_0$. 
Any set consisting of at most one such vertex from each brick of $W_0$ is a set of \emph{nails} of $W_0$.
If $W_0$ is a $1$-contained subwall of a wall~$W$,
then we can define the nails of $W_0$ in a more natural way. 
Between any two consecutive vertices in the top row that are adjacent to a vertex in $W_0$ outside the top row, 
there is exactly one branch vertex of $W$.
In a $t$-contained subwall (for $t\geq 1$) we always assume that the nails are chosen such that they are branch vertices of the ambient wall.

A \emph{linkage} $L$ of a wall $W$ is a set of disjoint paths 
such that each path in $L$ starts and ends at two distinct nails of $W$ and is otherwise disjoint from $W$. We refer to $|L|$ as the 
\emph{size} of $L$.
The top row of a wall imposes a natural ordering $<$ on the vertices in the top row (in fact, two; we fix one). 
We say that a vertex in the top row is to the \emph{left} of another vertex in the top row if it has a smaller order;
otherwise it is to the \emph{right} (provided the two vertices are distinct).
Thus, each linkage path has a left and a right endvertex.
A tuple $(L_1,\ldots,L_n)$ of linkages is \emph{path-disjoint} if 
for all distinct $i,j\in[n]$, no path in $L_i$ intersects any path in $L_j$.

\begin{figure}[htb]
\centering
\begin{tikzpicture}

\tikzstyle{oben}=[line width=1.5pt,double distance=1.2pt,draw=white,double=black]

\begin{scope}
\def\radius{0.25}
\def\step{0.3}
\draw[edg] (0,0) -- (3,0);
\draw[edg] (0.2,0) -- ++(0,0.2) arc (180:0:\radius) -- ++(0,-0.2);
\draw[edg] (0.2+2*\radius+\step,0) -- ++(0,0.2) arc (180:0:\radius) -- ++(0,-0.2);
\draw[edg] (0.2+4*\radius+2*\step,0) -- ++(0,0.2) arc (180:0:\radius) -- ++(0,-0.2);
\node at (2.7,0.3) {{\bf \dots}};
\node at (1.5,-0.3) {(a) in-series};
\end{scope}

\begin{scope}[shift={(3.5,0)}]
\draw[oben] (0.2,0) arc (180:0:0.75);
\draw[oben] (0.2+0.3,0) arc (180:0:0.75);
\draw[oben] (0.2+2*0.3,0) arc (180:0:0.75);
\draw[edg] (0,0) -- (3,0);
\node at (2.7,0.3) {{\bf \dots}};
\node at (1.5,-0.3) {(b) crossing};
\end{scope}

\begin{scope}[shift={(7,0)}]
\draw[edg] (0,0) -- (3,0);
\draw[edg] (0.2,0) arc (180:0:1.3);
\draw[edg] (0.2+0.3,0) arc (180:0:1.0);
\draw[edg] (0.2+2*0.3,0) arc (180:0:0.7);
\node at (1.3,0.3) {{\bf \dots}};
\node at (1.5,-0.3) {(c) nested};
\end{scope}

\end{tikzpicture}
\caption{The three types of pure linkages}\label{linkagefig}
\end{figure}
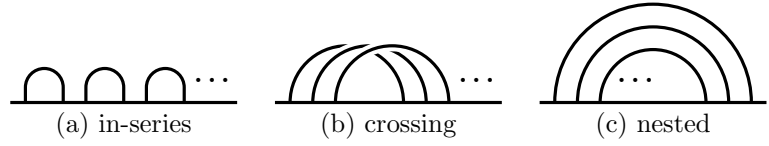

Let $L$ be a linkage, let $P_1,P_2\in L$, and let $\ell_i$ and $r_i$ be the left and right endvertices, respectively, 
of $P_i$ for each $i\in [2]$.
Then $L$ is \emph{nested} if for all choices of $P_1,P_2\in L$, it holds that $\ell_1<\ell_2$ implies $\ell_2<r_2<r_1$. It is \emph{crossing} if for all choices of $P_1,P_2\in L$, it holds that $\ell_1<\ell_2$ implies $\ell_2<r_1<r_2$. Lastly, it is \emph{in-series} if $\ell_1<\ell_2$ implies $\ell_1<r_1<\ell_2<r_2$.
The linkage $L$ is \emph{pure} if it is nested, crossing, or in-series.

\section{Paths and subwalls}

Given a wall and a set of disjoint paths $\cP$ that are connected to the wall in a certain way,
the following lemma shows that, 
without additional assumptions, we can modify $\cP$ to obtain a set of paths $\cP'$ with the same properties that intersect the wall only in a small part.

\begin{lemma}\label{reroutelem}
Let $W$ be a wall, and let $B$ be its set of branch vertices. Let $A$ be a vertex set, 
and let $\mathcal P$ be a set of disjoint $A$--$B$-paths, or a set of disjoint
$B$--$A$--$B$-paths. Then there is a set $\mathcal P'$ 
of disjoint paths with $|\mathcal P'|=|\mathcal P|$
such that 
\begin{enumerate}[label=\rm(\roman*)]
\item whenever a subwall $W'$ of $W$ does not contain any endvertex of any path in $\mathcal P'$,
$W'$ is also disjoint from every path in $\mathcal P'$; and
\item if the paths in $\mathcal P$ are $A$--$B$-paths/$B$--$A$--$B$-paths, then the paths 
in $\mathcal P'$ are $A$--$B$-paths/$B$--$A$--$B$-paths, too.
\end{enumerate}
\end{lemma}
\begin{proof}
Depending on the nature of $\mathcal P$, choose $\mathcal P'$
among all sets of $|\mathcal P|$ disjoint $A$--$B$-paths, or among all sets of $|\mathcal P|$ disjoint
$B$--$A$--$B$-paths
such that
\[
\sum_{P\in\mathcal P'}|E(P)\sm E(W)|
\]
is minimal.

Let $W'$ be a subwall of $W$ that is disjoint from the set of endvertices of $\mathcal P'$
but suppose that some path $P$ in $\mathcal P'$ intersects $W'$. 
Then there is a $B$-path $Q$ within $W'$ such that $P$ meets the interior of $Q$; let $b$ be one of the endvertices of~$Q$. 
Starting at $b$, let $x$ be the first vertex along $Q$ that is met by
some path in $\mathcal P'$, and let this path be $P^*$.
Denote the endvertices of $P^*$ by $p$ and $q$.
Note that the paths $P^*_1=pP^*xQb$ and $P^*_2=bQxP^*q$ are both disjoint from $\mathcal P'\sm\{P^*\}$
and that both $P^*_1$ and $P^*_2$ have fewer edges outside $W$ than $P^*$ --- the latter is due to the
fact that $p,q$ cannot lie in~$Q$ because $W'$ is disjoint from the endvertices of every path in $\mathcal P'$
and because $P^*$ cannot meet the endvertices of $Q$.

If $\mathcal P$ is a set of $A$--$B$-paths,
then one of the paths $P^*_1$ and $P^*_2$ is, or contains, an $A$--$B$-path; and 
if $\mathcal P$ is a set of $B$--$A$--$B$-paths,
then one of $P^*_1$ and $P^*_2$ is a $B$--$A$--$B$-path. In each case we obtain a contradiction to 
the minimality of $\mathcal P'$ by replacing $P^*$ with one of the paths $P^*_1$ or $P^*_2$.
\end{proof}

Lemma~\ref{subwalllem} reiterates the idea of Lemma~\ref{reroutelem};
if a wall is large enough, then paths can be rerouted so that they avoid a large subwall altogether.

\begin{lemma}\label{subwalllem}
Let $n,t',c,f$ be positive integers, and let $t\geq 100n^2f^2t'+2c$.
Let~$\cA$ be a set of $n$ vertex sets, let $W$ be a wall of size at least~$t$, and
let $B$ be the set of branch vertices of $W$.
Then there is a subwall $W'$ of $W$ of size at least~$t'$ that is $c$-contained in $W$
such that
\begin{enumerate}[label=\rm(\roman*)]
\item if, for no $A\in\mathcal A$, there is a vertex set of at most $f$ vertices that meets every
$A$--$B$-path, then, for each $A\in\mathcal A$, there exists a set $\mathcal P_A$ of $f$ disjoint $A$--$B$-paths
such that $W'$ is disjoint from every path in $\bigcup_{A\in\mathcal A}\mathcal P_A$; and
\item if, for no $A\in\mathcal A$, there is a vertex set of at most $2f$ vertices that meets every
$B$--$A$--$B$-path, then, for each $A\in\mathcal A$, there exists a set $\mathcal P_A$ of $f$ disjoint $B$--$A$--$B$-paths
such that $W'$ is disjoint from every path in $\bigcup_{A\in\mathcal A}\mathcal P_A$.
\end{enumerate}
Moreover, $W'$ can be chosen such that no row or column of $W$ that contains an endvertex
of a path in $\bigcup_{A\in\mathcal A}\mathcal P_A$ intersects $W'$.
\end{lemma}
\begin{proof}
For (i), we use Menger's theorem 
to obtain, for every $A\in\mathcal A$, a set $\mathcal P_A$ of $f$ disjoint $A$--$B$-paths.
For~(ii), we use Theorem~\ref{STSpath} to ensure that, for every $A\in\mathcal A$,
there is a set $\mathcal P_A$ of $f$ disjoint $B$--$A$--$B$-paths. 
In both cases, we may assume that, for every $A\in\mathcal A$, 
the set $\mathcal P_A$ is chosen as in Lemma~\ref{reroutelem}.

Let $U$ be the set of endvertices of the paths in $\bigcup_{A\in\mathcal A}\mathcal P_A$,
and observe that 
\[
|U|\leq 2\Big|\bigcup_{A\in\mathcal A}\mathcal P_A\Big|\leq 2nf.
\]
Put $\ell=4nf+1$.
As $t\geq 100n^2f^2t'+2c> 2\ell^2t'+2c$,
it follows that $W$ contains a set $\mathcal W$ of 
 $\ell^2$ disjoint subwalls of size $t'$ that are each $c$-contained in $W$.
Indeed, the set $\mathcal W$ can easily be obtained by splitting the rows and columns of $W$ 
(except for the first and last $c$ rows and columns that we exclude) in a grid-like fashion. 
Furthermore, this choice ensures that
each row and column of $W$ meets at most $\ell$ of the subwalls in~$\mathcal W$.
Each vertex in $U$ lies in at most one row and one column. Thus, if we exclude from $\mathcal W$
the subwalls that meet some row or column that contains a vertex from $U$, 
we exclude at most $2|U|\ell\leq 4nf \ell<\ell^2$ of the walls in $\mathcal W$. As $|\mathcal W|=\ell^2$,
we may choose one as $W'$ that is disjoint from any row or column 
that intersects $U$. Then, by Lemma~\ref{reroutelem},
$W'$ is also disjoint from every path in $\bigcup_{A\in\mathcal A}\mathcal P_A$.
\end{proof}

Let us call the induced subgraph of two consecutive columns, together with the two vertices of degree $2$ that have their neighbours in these columns, a \emph{brick column} of the elementary wall,
and let us call the largest $2$-connected subgraph in the induced subgraph of two consecutive rows a \emph{brick row}; see Figure~\ref{figbrickcol}.
A brick column or brick row in a wall is just the subdivision of a brick column or brick row in the underlying elementary wall.
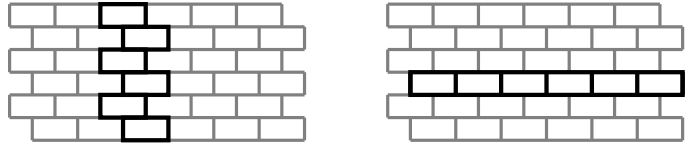
\begin{figure}[htb]
\centering
\begin{tikzpicture}

\tikzstyle{smallvx}=[thick,circle,inner sep=0.cm, minimum size=2mm, fill=white, draw=black]
\tikzstyle{hedge}=[very thick,dunkelgrau]

\def\wallheight{6}
\def\brickheight{0.3}

\pgfmathtruncatemacro{\lastrow}{\wallheight}
\pgfmathtruncatemacro{\penultimaterow}{\wallheight-1}
\pgfmathtruncatemacro{\lastrowshift}{mod(\wallheight,2)}
\pgfmathtruncatemacro{\lastx}{2*\wallheight+1}

\draw[hedge] (\brickheight,0) -- (2*\wallheight*\brickheight+\brickheight,0);
\foreach \i in {1,...,\penultimaterow}{
  \draw[hedge] (0,\i*\brickheight) -- (2*\wallheight*\brickheight+\brickheight,\i*\brickheight);
}
\draw[hedge] (\lastrowshift*\brickheight,\lastrow*\brickheight) to ++(2*\wallheight*\brickheight,0);

\foreach \j in {0,2,...,\penultimaterow}{
  \foreach \i in {0,...,\wallheight}{
    \draw[hedge] (2*\i*\brickheight+\brickheight,\j*\brickheight) to ++(0,\brickheight);
  }
}
\foreach \j in {1,3,...,\penultimaterow}{
  \foreach \i in {0,...,\wallheight}{
    \draw[hedge] (2*\i*\brickheight,\j*\brickheight) to ++(0,\brickheight);
  }
}

\draw[ultra thick, black] (\brickheight,3*\brickheight) rectangle  (\lastx*\brickheight,2*\brickheight); 
\foreach \j in {1,3,...,\lastx}{
  \draw[ultra thick,black] (\j*\brickheight,2*\brickheight) -- ++(0,\brickheight);
}

\begin{scope}[shift={(-5,0)}]

\draw[hedge] (\brickheight,0) -- (2*\wallheight*\brickheight+\brickheight,0);
\foreach \i in {1,...,\penultimaterow}{
  \draw[hedge] (0,\i*\brickheight) -- (2*\wallheight*\brickheight+\brickheight,\i*\brickheight);
}
\draw[hedge] (\lastrowshift*\brickheight,\lastrow*\brickheight) to ++(2*\wallheight*\brickheight,0);

\foreach \j in {0,2,...,\penultimaterow}{
  \foreach \i in {0,...,\wallheight}{
    \draw[hedge] (2*\i*\brickheight+\brickheight,\j*\brickheight) to ++(0,\brickheight);
  }
}
\foreach \j in {1,3,...,\penultimaterow}{
  \foreach \i in {0,...,\wallheight}{
    \draw[hedge] (2*\i*\brickheight,\j*\brickheight) to ++(0,\brickheight);
  }
}

\foreach \i in {0,...,\penultimaterow}{
  \pgfmathsetmacro{\flip}{mod(\i,2)}
  \draw[ultra thick,black] (-\flip*\brickheight+5*\brickheight,\i*\brickheight) rectangle ++(2*\brickheight,\brickheight);
}

\end{scope}

\end{tikzpicture}
\caption{A brick column and a brick row in a wall}\label{figbrickcol}
\end{figure}

Lemma~\ref{brickslem} provides a particular $2$-connected subgraph in a wall, which we need in the proof of Lemma~\ref{connlemX}.

\begin{lemma}\label{brickslem}
Let $r$ be a positive integer, and let $c\geq 10r$. 
Suppose $W$ is a wall, $X$ is a vertex set with $|X|\leq r$, and $W'$ is a $c$-contained subwall of $W$ of size at least~$c$.
Then there exists a $2$-connected subgraph $C$ of $W$ such that 
\begin{enumerate}[label=\rm(\roman*)]
\item $C$ is disjoint from $X$;
\item $C$ meets $W'$ only in its top row and contains at least two nails; and
\item $C$ meets every row and every column of $W$.
\end{enumerate}
\end{lemma}
\begin{proof}
As the subwall $W'$ of $W$ has size at least $10 r\geq 10|X|$,
there is a brick column $C'$ of $W$ that, restricted to $W'$,
is also a brick column of $W'$ and that is disjoint from $X$.
In the graph obtained from $C'$
by removing all of $W'$ except for the first row, let $C''$ be the component that meets 
the top row of $W$. 
Observe that~$C''$ is $2$-connected and contains two nails of $W'$.

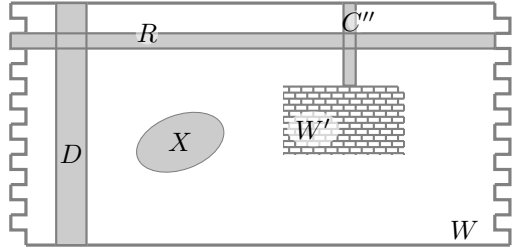
\begin{figure}[htb]
\centering
\begin{tikzpicture}[scale=0.8]

\tikzstyle{smallvx}=[thick,circle,inner sep=0.cm, minimum size=1mm, fill=white, draw=black]
\tikzstyle{hedge}=[very thick]

\tikzstyle{hvertex}=[thick,circle,inner sep=0.cm, minimum size=1.8mm, fill=white, draw=black]
\tikzstyle{wed}=[hedge,color=dunkelgrau]
\tikzstyle{jumps}=[ultra thick, white, double distance=1pt, double=black,bend left=90]
\tikzstyle{brickline}=[thick,hellgrau,fill=hellgrau]
\tikzstyle{lineb}=[thick,dunkelgrau]

\def\bh{0.25}
\def\downhook{-- ++(0,-\bh) -- ++(\bh,0) -- ++(0,-\bh) -- ++(-\bh,0)}
\def\lasthook{-- ++(0,-\bh) -- ++(\bh,0) -- ++(0,-\bh)}
\def\wallh{8}
\def\rightx{8}

\begin{scope}[on background layer]

\foreach \j in {2,...,\wallh}{
  \draw[wed] (0,-\j*2*\bh+4*\bh) \downhook;
  \draw[wed] (\rightx,-\j*2*\bh+4*\bh) \downhook;
}
\draw[wed] (0,-\wallh*2*\bh+2*\bh) \lasthook;
\draw[wed] (\rightx,-\wallh*2*\bh+2*\bh) \lasthook;

\draw[wed] (0,0) -- (\rightx,0);
\draw[wed] (\bh,-\wallh*2*\bh) -- (\rightx+\bh,-\wallh*2*\bh);
\end{scope}

\draw[brickline] (0,-3*\bh) rectangle ++(\wallh*\bh*4,\bh);
\draw[brickline] (3*\bh,0) rectangle ++(2*\bh,-\wallh*\bh*2);

\draw[brickline] (22*\bh,0) rectangle (22*\bh+0.2,-1.36);

\draw[lineb] (0,-3*\bh) rectangle ++(\wallh*\bh*4,\bh);
\draw[lineb] (3*\bh,0) rectangle ++(2*\bh,-\wallh*\bh*2);

\draw[lineb] (22*\bh,0) rectangle (22*\bh+0.2,-1.365);

\node at (4*\bh,-10*\bh) {$D$};
\draw[fill=white,white,opacity=0.5] (9*\bh,-2*\bh) circle (5pt);
\node  at (9*\bh,-2*\bh) {$R$};
\draw[fill=white,white,opacity=0.5] (23*\bh,-0.3) circle (5pt);
\node at (23*\bh,-0.3) {$C''$};

\node at (\wallh*4*\bh-2*\bh,-\wallh*2*\bh+\bh) {$W$};

\begin{scope}[shift={(2.8,-2.3)},rotate=20]
\draw[dunkelgrau,fill=hellgrau] (0,0) ellipse [x radius=0.75, y radius=0.45];
\end{scope}
\node at (2.8,-2.3) {$X$};

\clip[] (4.5,-2.5) rectangle (6.5,-1);

\begin{scope}[shift={(4,-2.58)}]
\tikzstyle{wed}=[thick,color=dunkelgrau]
\def\wallheight{12}
\def\brickheight{0.1}

\pgfmathtruncatemacro{\lastrow}{\wallheight}
\pgfmathtruncatemacro{\penultimaterow}{\wallheight-1}
\pgfmathtruncatemacro{\lastrowshift}{mod(\wallheight,2)}
\pgfmathtruncatemacro{\lastx}{2*\wallheight+1}

\draw[wed] (\brickheight,0) -- (2*\wallheight*\brickheight+\brickheight,0);
\foreach \i in {1,...,\penultimaterow}{
  \draw[wed] (0,\i*\brickheight) -- (2*\wallheight*\brickheight+\brickheight,\i*\brickheight);
}
\draw[wed] (\lastrowshift*\brickheight,\lastrow*\brickheight) to ++(2*\wallheight*\brickheight,0);

\foreach \j in {0,2,...,\penultimaterow}{
  \foreach \i in {0,...,\wallheight}{
    \draw[wed] (2*\i*\brickheight+\brickheight,\j*\brickheight) to ++(0,\brickheight);
  }
}
\foreach \j in {1,3,...,\penultimaterow}{
  \foreach \i in {0,...,\wallheight}{
    \draw[wed] (2*\i*\brickheight,\j*\brickheight) to ++(0,\brickheight);
  }
}

\draw[fill=white,white,opacity=0.7] (1,0.5) ellipse [x radius=12pt, y radius=8pt];
\node at (1,0.5) {$W'$};

\end{scope}

\end{tikzpicture}
\caption{The subgraph $C=C''\cup R\cup D$ of Lemma~\ref{brickslem}.}\label{framefig}
\end{figure}

Since $W'$ is $10r$-contained in $W$,
there is a brick row $R$ of $W$ contained in the first $10r$ rows of $W$ that is 
disjoint from $X$, and hence from $W'$, too. 
Note that the union of $R$ and $C''$ is $2$-connected. 
We also pick a brick column~$D$ of~$W$ that is disjoint from~$W'$ and from~$X$. 
Again this is possible because $|X|\leq r$ and because we can choose it to be 
contained in the first $c$ columns.
We set $C=C''\cup R\cup D$, which completes the proof.
\end{proof}

\begin{lemma}\label{brickcollem}
Let $U$ be a subset of the set of branch vertices of a wall $W$, and assume that 
$|U|\geq 10m^2$ for some integer $m\geq 9$. Then $W$ either contains at least
$m$ disjoint brick rows that each meet $U$ or it contains at least $m$ disjoint brick columns
that each meet $U$.
\end{lemma}
\begin{proof}
Assume first that no brick row of $W$ contains at least $8m+2$ vertices from $U$.
The branch vertices of the wall can be covered by disjoint brick rows plus, potentially, one (ordinary) row.
As that ordinary row also contains fewer than $8m+2$ vertices from $U$, 
it follows that at least the following number of the disjoint brick rows meet $U$:
\[
\frac{|U|-(8m+2)}{8m+2} \geq \frac{9m^2+m^2-9m}{9m} \geq m,
\]
where we have used that $m\geq 9$.

Second, assume that some brick row, $R$ say, meets at least $8m+2$ vertices from $U$.
The branch vertices of the wall $W$ can be covered by disjoint brick columns, except
for possibly the last (ordinary) column. As the last (ordinary) column meets $R$
in two branch vertices, it follows that there is a set $\mathcal C$ of disjoint
brick columns such that at least $8m+2-2=8m$ of the vertices in $U$ lie in 
the intersection of $R$ and some brick column in $\mathcal C$. Since any 
brick column meets $R$ in eight branch vertices of the wall, it follows 
that at least $m$ of the disjoint brick columns in $\mathcal C$ contain 
some vertex from $U$. 
\end{proof}

We proceed with the main result of this section.
We write \emph{$A$-linkage} for an $\{A\}$-linkage, 
and \emph{$A$--$W$-connector} for an $\{A\}$--$W$-connector.

\begin{lemma}\label{connlemX}
For all positive integers $n,s',t'$, 
there are integers $t=t(n,s',t')$ and $f=f(s')$ such that the following holds.
If $\mathcal A$ is a set of $n$ vertex sets, and~$W$ is a wall of size at least~$t$
with $B$ as its set of branch vertices, then
\begin{enumerate}[label=\rm(\roman*)]
\item if for no $A\in\mathcal A$, there is a vertex set of size at most~$f$ 
that meets every $A$--$B$-path, 
then there is a subwall $W'$ of size at least~$t'$
and for every $A\in\mathcal A$ an $A$--$W'$-connector of size~$s'$; and
\item if for no $A\in\mathcal A$, there is a vertex set of size at most~$f$ 
that meets every $B$--$A$--$B$-path, 
then there is a subwall $W'$ of size at least~$t'$
and for every $A\in\mathcal A$ an $A$-linkage of $W'$ of size~$s'$.
\end{enumerate}
\end{lemma}
While the lemma looks very similar to Theorems~\ref{connlem} and~\ref{linkagelem}, 
note that there is a significant difference: the lemma only yields an $A$-linkage/an $A$--$W'$-connector
for each $A\in\mathcal A$, but not an $\mathcal A$-linkage/$\cA$--$W'$-connector.
In particular, the paths in the $A_1$--$W'$-connector could intersect the paths in the $A_2$--$W'$-connector.
\begin{proof}[Proof of Lemma~\ref{connlemX}]
Given $n,s',t'$, we set 
\[
f=2\cdot 10^7s'^4,\quad c=20s'\quad \text{and} \quad t\geq 100n^2f^2t'+2c.
\]

Now let a graph $G$ with a set $\cA$ of $n$ vertex sets and a wall $W$ of size at least~$t$
with a set of branch vertices $B$ be given, and assume that the precondition
of~(i) or of~(ii) is satisfied. 
We will give the proof for~(ii) as it is slightly more complicated and only indicate 
where and how the arguments need to be adapted for~(i). 

As there is no set of $f$ vertices that meets every $B$--$A$--$B$-path for any $A\in\mathcal{A}$ (or every $A$--$B$-path in the case of (i)), we may start by 
applying Lemma~\ref{subwalllem} to obtain a $c$-contained subwall $W'$ of size at least $t'$
and, for each $A\in\mathcal A$, a set $\mathcal P_A$ of disjoint $B$--$A$--$B$-paths of size~$10^7s'^4$
such that $W'$ is 
disjoint from every path in $\mathcal P_A$ for every $A\in\mathcal A$
and is also disjoint from every row or column of $W$ that contains an endvertex of a path 
in $\bigcup_{A\in\mathcal A}\mathcal P_A$.
Let the set of nails of $W'$ be $N'$.

Suppose that there is some $A\in\mathcal A$ for which there is no $A$-linkage of $W'$ of size~$s'$.
By Theorem~\ref{STSpath}, there is thus a vertex set $X$ of size smaller than $2s'$
such that $G-(W'-N')$ does not contain any $N'$--$A$--$N'$-path.
(For~(i), we simply use Menger's theorem here.)

Next, we apply Lemma~\ref{brickslem} to $W$, $W'$ and $X$, with $r=2s'$. We obtain a $2$-connected
subgraph $C$ of $W$ that is disjoint from $X$, that meets $W'$ only in its top row, where it 
contains at least two nails, and that intersects every row and column of $W$.

As a convenience, let us subsume brick rows and brick columns under the common name \emph{brick lines}. 
For each path in $\mathcal P_A$,
pick one of the endvertices and call it the \emph{first endvertex}; 
the other endvertex then becomes the \emph{second endvertex}.
Let $U_1$ be the set of all first endvertices of the paths in $\mathcal P_A$, and note that $|U_1|=|\mathcal P_A| = 10^7s'^4$. 
We apply Lemma~\ref{brickcollem} to $U_1$
and $W$, and obtain a set of at least $10^3s'^2$ disjoint brick lines $\mathcal B_1$ that each meet $U_1$.
For each brick line $L$ in $\mathcal B_1$, choose a path in $\mathcal P_A$ such that $L$ contains the first endvertex of the path;
let the set of these paths be $\mathcal P^{(1)}$ and note that $|\mathcal P^{(1)}|= 10^3s'^2$.

Defining $U_2$ to be the set of second endvertices of the paths in $\mathcal P^{(1)}$, 
we apply Lemma~\ref{brickcollem} again,
and obtain a set $\mathcal B_2$ of $10s'$ disjoint brick lines that each meet~$U_2$. 
Let $\mathcal P^{(2)}$ be the set 
of those paths in $\mathcal P^{(1)}$ whose second endvertex is met by some brick line in $\mathcal B_2$. 
Of the (disjoint) brick lines in $\mathcal B_1$, fewer than $2s'$ meet $X$, as $|X|< 2s'$, and similarly 
fewer than $2s'$ of the brick lines in $\mathcal B_2$ meet $X$.
As $|\mathcal B_2|\geq 10s'$, there is thus a path $P$ in $\mathcal P^{(2)}$ (and thus in $\mathcal P_A$) such that 
there is a brick line $L_1$ that contains its first endvertex, a brick line $L_2$ that contains its second endvertex and 
such that neither $L_1$ nor $L_2$ meets $X$.
In fact, there are at least $6s' > 5$ such paths $P$.
For at least one of these five paths,
neither $L_1$ nor $L_2$ intersects
$W'$ (observe that the choice of $W'$ only guarantees that each row or column that contains an endvertex of $P$ is disjoint from $W'$ --- however, $L_1$ and $L_2$ are brick lines).

As $C$ meets every row and every column of $W$, it follows that $C\cup L_1\cup L_2\cup P$ 
is a $2$-connected subgraph of $W$ that is
disjoint from $X$ and meets $W'$ only in the top row of $W'$. Since $C$ contains two nails of $W'$,
we find in $C\cup L_1\cup L_2\cup P$ an $N'$--$A$--$N'$-path that is disjoint from $W'-N'$. This, however, 
contradicts the choice of $X$.
\end{proof}

\section{Disentangling the paths}

In the previous section, Lemma~\ref{connlemX} already provides an important step towards our goal.
However, for different $A\in \cA$,
the corresponding sets of paths may interfere.
In this section, we resolve this problem.

\begin{lemma}\label{ABsortlem}
Let $n,t$ be positive integers. Let $A_1,\ldots, A_{n}$ be vertex sets in a graph~$G$, and suppose
that, for every $i\in[n]$, there is a set $\mathcal Q_i$ of at least~$2^{n}t$ disjoint $A_i$--$B$-paths.
Then there are sets $\mathcal P_i$ of $A_i$--$B$-paths for every $i\in[n]$
such that 
\begin{enumerate}[label=\rm(\roman*)]
\item the paths in $\bigcup_{i=1}^n\mathcal P_i$ are pairwise disjoint;
\item $|\mathcal P_i|\geq t$ for every $i\in[n]$; and
\item every path $P\in\bigcup_{i=1}^n\mathcal P_i$ is contained in $\bigcup_{i=1}^n\bigcup_{Q\in\mathcal Q_i}Q$.
\end{enumerate}
\end{lemma}
For $n=2$ the lemma is proved in~\cite{BU18}. The version here can be obtained
from that result by straightforward induction.
Alternatively, for a direct proof, one may adapt the arguments of the proof of the next lemma, 
which is basically a reformulation of Lemma~\ref{ABsortlem} for $B$--$A_i$--$B$-paths. 
There is a small twist, though: we also specify an $n$-th path system $\mathcal Q_n$ 
that we do not allow to be rerouted. Rather, we claim that we may find a suitable subset of the paths in $\mathcal Q_n$.

\begin{lemma}\label{disjlink}
Let $n,t$ be positive integers. Let $A_1,\ldots, A_{n-1}$ be vertex sets in a graph $G$, and suppose
that, for every $i\in[n-1]$, there is a set $\mathcal Q_i$ of at least~$3^{i-1}t$ disjoint $B$--$A_i$--$B$-paths,
and let $\mathcal Q_n$ be a set of at least $3^{n-1}t$ pairwise disjoint paths such that each path has an endvertex in $B$
but is disjoint from $B$ in its interior. 
Then there are sets $\mathcal P_i$ of $B$--$A_i$--$B$-paths for every $i\in[n-1]$
and a set $\cP_n$ of paths such that 
\begin{enumerate}[label=\rm(\roman*)]
\item the paths in $\bigcup_{i=1}^n\mathcal P_i$ are pairwise disjoint;
\item $|\mathcal P_i|\geq t$ for every $i\in[n]$;
\item every path $P\in\bigcup_{i=1}^n\mathcal P_i$ is contained in $\bigcup_{i=1}^n\bigcup_{Q\in\mathcal Q_i}Q$; and
\item $\mathcal P_n\subseteq\mathcal Q_n$, and $|\mathcal P_n|\geq t+|\mathcal Q_n|- 3^{n-1}t$.
\end{enumerate}
\end{lemma}
\begin{proof}
We prove the statement by induction on $n$. For $n=1$, the statement is clearly true. So let $n\geq 2$. 
For $i \in [n]\setminus \{1\}$, the set $\cQ_i$ contains $3^{i-2}(3t)$ paths. With a simple index shift we can apply the induction to $\cQ_2,\ldots, \cQ_n$
in order to find sets $\cR_2,\ldots,\cR_{n-1}$
of size $3t$ consisting of $B$--$A_i$--$B$-paths for $i \in [n-1]\setminus \{1\}$
and a set of paths $\cR_n$
such that all paths 
in $\bigcup_{i=2}^n\cR_i$ are pairwise disjoint, contained in $\bigcup_{i=2}^n \bigcup_{Q\in\cQ_i} Q$, 
and such that $\mathcal R_n\subseteq\mathcal Q_n$ has size at least 
\begin{equation}\label{Rnsize}
|\mathcal R_n|\geq 3t+|\mathcal Q_n|-3^{n-2}(3t) = 3t+|\mathcal Q_n|-3^{n-1}t.
\end{equation}

Let $\cQ_1=\{Q_1,\ldots,Q_{t}\}$ and let $a_1,\ldots,a_t\in A_1$ be such that $a_\ell\in V(Q_\ell)$ for $\ell\in [t]$.
If possible, let $a_\ell$ be an endvertex of $Q_\ell$.
Let $\{Q_\ell^1,Q_\ell^2 : \ell \in [t]\}$ be a set of internally disjoint paths such that 
\begin{enumerate}[label=\rm(\roman*)]
\item  for each $\ell\in[t]$, 
if $a_\ell\notin B$ then $Q_\ell^1$ and $Q_\ell^2$ are $a_\ell$--$B$-paths; and if $a_\ell\in B$ 
then $Q_\ell^1=\{a_\ell\}$ is a trivial path and $Q_\ell^2$ is an $a_\ell$--$(B\setminus \{a_\ell\})$-path; \label{Q1}
\item any two distinct paths $Q_\ell^j$ and $Q_{\ell'}^{j'}$ may meet only in $a_\ell$ and only if $\ell=\ell'$; and\label{Q2}
\item $Q_\ell^1,Q_\ell^2\subseteq\bigcup_{i=1}^n \bigcup_{Q\in\cQ_i} Q$.\label{Q3}
\end{enumerate}

Note that splitting each path $Q_\ell\in\cQ_1$ at $a_\ell$ into two subpaths yields such a set. 
Among all such sets, choose a set $\cR_1=\{Q_\ell^1,Q_\ell^2 : \ell \in [t]\}$ such that
\begin{equation}
\sum_{Q\in\mathcal R_1}\Big|E(Q)\setminus \bigcup_{i=2}^n\bigcup_{R\in\mathcal R_i}E( R)\Big|\label{Rmin}
\end{equation}
is minimum. 
Observe that the size of $\cR_1$ is $2t$.
Our aim is to show that for each $i \in [n]\setminus \{1\}$, there are sufficiently many paths in $\cR_i$ that are disjoint from $\cR_1$. 

For $\ell\in [t]$ and $j\in [2]$, let $b^j_\ell$ be such that $Q_\ell^j$ is an $a_\ell$--$b^j_\ell$-path.
We claim that
\begin{equation}\label{path_ending}
	\begin{minipage}[c]{0.8\textwidth}\em
		every path $R\in\bigcup_{i=2}^n\cR_i$ that intersects some path in $\cR_1$ 
		has one of the $b^j_\ell$ as an endvertex. 
	\end{minipage}\ignorespacesafterend 
\end{equation} 
Suppose for a contradiction that there is a path $R\in\bigcup_{i=2}^n\cR_i$ that intersects some path in $\cR_1$ 
but that does not end in any $b^j_\ell$.
Observe that $R$ cannot intersect any trivial path in $\cR_1$, as it would otherwise end there (this is a vertex of $B$,
and $R$ is disjoint from $B$ in its interior) 
and hence end in some $b^j_\ell$.
As every path in $\bigcup_{i=2}^n\mathcal R_i$ has an endvertex in $B$, 
one of the endvertices $y$ of $R$ lies in $B$. 
Let $x$ be the vertex of $R\cap \bigcup_{T\in\cR_1} T$ that is closest to $y$ on $R$,
and let $Q_\ell^j\in\cR_1$ be the path that contains $x$. 

Suppose that $x=y$.
As $y\in B$, the path~$Q_\ell^j$ has to end in $y$, by~\ref{Q1}, and as $y\neq b_\ell^j$, we have $y=a_\ell$.
However, then $a_\ell\in B$ as $y\in B$, and thus~\ref{Q1} implies
$y=b_\ell^{3-j}$, which contradicts the choice of $R$.
We deduce that $x\neq y$.

By the choice of $x$, the subpath $xRy$ is disjoint from any path in $\cR_1\setminus \{Q_\ell^j\}$ and $Q_\ell^j$ intersects this subpath only in $x$. 
We also note that 
\begin{equation}\label{notx}
x\neq b^j_\ell.
\end{equation}
Indeed, $x=b^j_\ell$ would mean that $x\in B$, which then would entail that the path~$R$ contains a vertex of $B$ in its interior,
which is impossible.

The concatenation $P=a_\ell Q_\ell^jxRy$ is an $a_\ell$--$(B\setminus \{a_\ell\})$-path.
The only path in $\cR_1\setminus \{Q_\ell^j\}$ that $P$ intersects is $Q_\ell^{3-j}$, and it does so precisely in $a_\ell$.

We replace $Q_\ell^j$ by $P$ in $\cR_1$ and denote the resulting set by $\cR_1'$.
Observe that $\cR_1'$ satisfies (i)--(iii).
Note that $xQ^j_\ell b^j_\ell\nsubseteq R$ as $R$ does not contain a vertex of $B$ in its interior (and $b_\ell^j\in B$). 
Thus, $xQ^j_\ell b^j_\ell$ contains an edge outside $\bigcup_{i=2}^n\bigcup_{R'\in\mathcal R_i}E(R')$. 

Therefore the path $P$ has fewer edges outside $\bigcup_{i=2}^n\bigcup_{R'\in\mathcal R_i}E(R')$ than $Q_\ell^j$,
which means that $\cR_1'$ is a better choice with respect to~\eqref{Rmin}
than $\mathcal R_1$. This contradiction proves~\eqref{path_ending}.

By~\eqref{path_ending}, it follows that whenever a path $R\in\bigcup_{i=2}^n\cR_i$ intersects some path in~$\cR_1$,
then it ends in some $b^j_\ell$. As there are only~$2t$ of the $b^j_\ell$, we see that at most~$2t$
of the paths in $\bigcup_{i=2}^n\cR_i$ may meet any path in~$\cR_1$ at all. 

Recall that $|\cR_i|= 3t$ for each $i \in [n]\setminus \{1\}$, and note~\eqref{Rnsize}.
Therefore, for each $i \in [n]\setminus \{1\}$, we find a subset $\cP_i$ of $\cR_i$ of cardinality $t$
such that each of the paths in $\cP_i$ is disjoint from every path in $\mathcal R_1$, 
and for $\mathcal R_n$ we even find $t+|\mathcal Q_n|-3^{n-1}t$ such paths. 
Let $\cP_1$ be the set of paths $P_\ell=Q_\ell^1\cup Q_\ell^2$. 
Note that $P_\ell$ is a $B$-path by~(i) and contains $a_\ell\in A_1$.  
So $\cP_1$ contains $t$ disjoint $B$--$A_1$--$B$-paths.
Thus, the paths $\cP_1,\ldots, \cP_n$ are as desired.
\end{proof}

With Lemmas~\ref{ABsortlem} and~\ref{disjlink} in hand, 
we can finish the proofs of Theorems~\ref{connlem}
and~\ref{linkagelem}.
\begin{proof}[Proof of Theorems~\ref{connlem} and~\ref{linkagelem}]
	Given $n$, $s'$ and $t'$, apply Lemma~\ref{connlemX} with $n$, $3^ns', t'$ in the roles of $n, s'$, and $t'$. 
	Then the resulting $A$--$W'$-connectors/$A$-linkages of~$W'$ can 
	be turned into an $\mathcal A$--$W'$-connector/$\mathcal A$-linkage of $W'$
	of size $s'$ with the help of Lemmas~\ref{ABsortlem} and~\ref{disjlink}.
\end{proof}

We may also combine Lemmas~\ref{ABsortlem} and~\ref{disjlink}. 
This is also essential for a combined version of Theorems~\ref{connlem} and~\ref{linkagelem} in Lemma~\ref{monsterlem}.

\begin{lemma}\label{combilem}
Let $n,t$ be positive integers.
Let $B$ be a set of vertices, 
let $\mathcal A$ be a set of $n$ vertex sets, and suppose that, for every $A\in\mathcal A$,
there is a set $\mathcal Q_A$ of at least $3^{2n}t$ disjoint paths
that are either all $A$--$B$-paths or all $B$--$A$--$B$-paths.
Then, for every $A\in\mathcal A$, there exists a set $\mathcal P_A$ of $t$ paths such that
\begin{enumerate}[label=\rm(\roman*)]
\item the paths in $\bigcup_{A\in\mathcal A}\mathcal P_A$ are pairwise disjoint;
\item for every $A\in\mathcal A$, the set $\mathcal P_A$ is a set of $A$--$B$-paths 
if $\mathcal Q_A$ is a set of $A$--$B$-paths, and $\mathcal P_A$ is a set of $B$--$A$--$B$-paths 
if $\mathcal Q_A$ is a set of $B$--$A$--$B$-paths; and
\item every path $P\in\bigcup_{A\in\mathcal A}\mathcal P_A$ lies in $\bigcup_{A\in\mathcal A}\bigcup_{Q\in\mathcal Q_A}Q$.
\end{enumerate}
\end{lemma}
\begin{proof}
Let $\mathcal A_1$ be the subset of the $A\in\mathcal A$ for which $\mathcal Q_A$ is a set
of $A$--$B$-paths, and let $\mathcal A_2$ be the subset of the $A\in\mathcal A$
for which $\mathcal Q_A$ is a set of $B$--$A$--$B$-paths.

We apply Lemma~\ref{ABsortlem} to $\mathcal A_1$ with $3^nt$
in the role of $t$. 
Denote the resulting sets of $3^nt$ paths by $\mathcal Q'_{A}$, $A\in\mathcal A_1$,
and put $\mathcal Q'=\bigcup_{A\in\mathcal A_1}\mathcal Q'_{A}$.

Next, we apply Lemma~\ref{disjlink} to the path sets $\mathcal Q_A$, $A\in\mathcal A_2$, together
with $\mathcal Q'$
in the role of the last path set. As a result, for each $A\in\mathcal A_2$, we obtain
a set $\mathcal P_{A}$ of~$t$ disjoint $B$--$A$--$B$-paths, and a set $\mathcal P\subseteq \mathcal Q'$ 
of at least $t+|\mathcal Q'|-3^{n-1}t$ paths such that all these paths are pairwise disjoint. 
Now, for each $A\in\mathcal A_1$, at least 
\[
|\mathcal Q'_{A}|-3^{n-1}t\geq t 
\]
of the $A$--$B$-paths in $\mathcal Q'_{A}$
are retained in $\mathcal P$; let a set of these be denoted by~$\cP_{A}$. 
Then the sets $\mathcal P_{A}$, $A\in\mathcal A$
are as desired.
\end{proof}

\section{Disentangling the endvertices}

The paths in Theorems~\ref{connlem} and~\ref{linkagelem} can even be arranged in a nicer way.
Let $W$ be a wall.
Let $\mathcal P=\{\mathcal P_1,\ldots,\mathcal P_n\}$ be a set such that 
for every $i\in [n]$,
the set $\mathcal P_i$ is either an in-series linkage of $W$ or an $A_i$--$W$-connector for some vertex set~$A_i$.
Then $\mathcal P$ is \emph{clean} if for any two distinct $\mathcal P_i$ and $\mathcal P_j$,
there are disjoint subpaths $P_i$ and $P_j$ of the top row of $W$ such that $P_i$ contains
all endvertices in $W$ of all the paths in~$\mathcal P_i$ and 
$P_j$~contains all the endvertices in~$W$ of all the paths in~$\mathcal P_j$. 
(That is, the paths $P_i$, $P_j$ are required to 
contain only those endvertices that actually lie in the top row.) 

\begin{lemma}\label{cleanlinesslem}
Let $W$ be a wall.
Let $\mathcal P=\{\mathcal P_1,\ldots,\mathcal P_n\}$ be a set such that for every $i\in [n]$,
the set $\mathcal P_i$ is either an in-series linkage of $W$ or an $A_i$--$W$-connector for some vertex set $A_i$.
If $|\mathcal P_i|\geq 2ns$ for every $i\in[n]$, 
then there are subsets $\mathcal P'_i\subseteq\mathcal P_i$
of size $s$ for every $i\in[n]$ such that $\{\mathcal P'_1,\ldots,\mathcal P'_n\}$
is clean.
\end{lemma}
\begin{proof}
For each $i\in[n]$, let $P_i$ be a subpath of the top row of $W$ that contains all 
the endvertices in $W$ of the paths in $\mathcal P_i$. Split $P_i$ into $2n$ subpaths $P_{i,1},\ldots,P_{i,2n}$
such that each $P_{i,j}$ contains the endvertices in $W$ of exactly $s$ of the paths in~$\mathcal P_i$.
(If $\mathcal P_i$ is a linkage, this is possible because it is then an in-series linkage.)
Note that we are done if we can pick a $j_i$ for each $i\in[n]$ such that all the paths 
\[
P_{1,j_1},\ldots, P_{n,j_n}
\]
are pairwise disjoint; but this can be done inductively: 
Pick $i\in [n],j\in [2n]$ such that $P_{i,j}$ does not properly contain some other path $P_{k,\ell}$; 
then $P_{i,j}$
is disjoint from at least $2n-2$ of the paths $P_{i',1},\ldots, P_{i',2n}$ for every $i'\in [n]$. 
Delete all the paths that intersect $P_{i,j}$ and iterate. 
\end{proof}

We also need the following lemma.

\begin{lemma}[Huynh, Joos and Wollan~\cite{HJW19}]\label{pureLink}
Let $L$ be a linkage of a wall $W$ of size $s^3$.
Then $L$ contains a pure linkage $L'$ of size $s$.
\end{lemma}

\begin{figure}[htb]
\centering
\begin{tikzpicture}

\tikzstyle{oben}=[line width=1.5pt,double distance=1.2pt,draw=white,double=black]
\tikzstyle{tinyvx}=[thick,circle,inner sep=0.cm, minimum size=1mm, fill=white, draw=black]
\tikzstyle{marked}=[line width=2pt,color=dunkelgrau]
\tikzstyle{wed}=[thick,color=dunkelgrau]

\def\wallheight{10}
\def\brickheight{0.2}

\pgfmathtruncatemacro{\lastrow}{\wallheight}
\pgfmathtruncatemacro{\penultimaterow}{\wallheight-1}
\pgfmathtruncatemacro{\lastrowshift}{mod(\wallheight,2)}
\pgfmathtruncatemacro{\lastx}{2*\wallheight+1}

\pgfmathsetmacro{\topy}{\wallheight*\brickheight}

\begin{scope}
\clip[] (\brickheight,\topy-0.75) rectangle (2*\wallheight*\brickheight-\brickheight,\topy+1);

\draw[oben] (3*\brickheight,\topy) arc (180:0:4*\brickheight);
\draw[oben] (5*\brickheight,\topy) arc (180:0:4*\brickheight);

\draw[wed] (\brickheight,0) -- (2*\wallheight*\brickheight+\brickheight,0);
\foreach \i in {1,...,\penultimaterow}{
  \draw[wed] (0,\i*\brickheight) -- (2*\wallheight*\brickheight+\brickheight,\i*\brickheight);
}
\draw[wed] (\lastrowshift*\brickheight,\lastrow*\brickheight) to ++(2*\wallheight*\brickheight,0);

\foreach \j in {0,2,...,\penultimaterow}{
  \foreach \i in {0,...,\wallheight}{
    \draw[wed] (2*\i*\brickheight+\brickheight,\j*\brickheight) to ++(0,\brickheight);
  }
}
\foreach \j in {1,3,...,\penultimaterow}{
  \foreach \i in {0,...,\wallheight}{
    \draw[wed] (2*\i*\brickheight,\j*\brickheight) to ++(0,\brickheight);
  }
}

\draw[edg] (3*\brickheight,\topy) -- ++(\brickheight,0) -- ++(0,-\brickheight);
\draw[edg] (5*\brickheight,\topy) -- ++(\brickheight,0) -- ++(0,-\brickheight);
\draw[edg] (11*\brickheight,\topy) -- ++(2*\brickheight,0);

\node[tinyvx] at (3*\brickheight,\topy){};
\node[tinyvx] at (5*\brickheight,\topy){};
\node[tinyvx] at (11*\brickheight,\topy){};
\node[tinyvx] at (13*\brickheight,\topy){};
\node[tinyvx] at (4*\brickheight,\topy-\brickheight){};
\node[tinyvx] at (6*\brickheight,\topy-\brickheight){};
\end{scope}

\begin{scope}[shift={(5,0)}]
\clip[] (\brickheight,\topy-0.75) rectangle (2*\wallheight*\brickheight-\brickheight,\topy+1.5);

\draw[oben] (3*\brickheight,\topy) arc (180:0:6*\brickheight);
\draw[oben] (5*\brickheight,\topy) arc (180:0:4*\brickheight);

\draw[wed] (\brickheight,0) -- (2*\wallheight*\brickheight+\brickheight,0);
\foreach \i in {1,...,\penultimaterow}{
  \draw[wed] (0,\i*\brickheight) -- (2*\wallheight*\brickheight+\brickheight,\i*\brickheight);
}
\draw[wed] (\lastrowshift*\brickheight,\lastrow*\brickheight) to ++(2*\wallheight*\brickheight,0);

\foreach \j in {0,2,...,\penultimaterow}{
  \foreach \i in {0,...,\wallheight}{
    \draw[wed] (2*\i*\brickheight+\brickheight,\j*\brickheight) to ++(0,\brickheight);
  }
}
\foreach \j in {1,3,...,\penultimaterow}{
  \foreach \i in {0,...,\wallheight}{
    \draw[wed] (2*\i*\brickheight,\j*\brickheight) to ++(0,\brickheight);
  }
}

\draw[edg] (3*\brickheight,\topy) -- ++(\brickheight,0) -- ++(0,-\brickheight);
\draw[edg] (5*\brickheight,\topy) -- ++(\brickheight,0) -- ++(0,-\brickheight);
\draw[edg] (13*\brickheight,\topy) -- ++(2*\brickheight,0);

\node[tinyvx] at (3*\brickheight,\topy){};
\node[tinyvx] at (5*\brickheight,\topy){};
\node[tinyvx] at (15*\brickheight,\topy){};
\node[tinyvx] at (13*\brickheight,\topy){};
\node[tinyvx] at (4*\brickheight,\topy-\brickheight){};
\node[tinyvx] at (6*\brickheight,\topy-\brickheight){};
\end{scope}

\end{tikzpicture}
\caption{How to turn a crossing or nested linkage into one that is in-series}\label{seriesfig}
\end{figure}
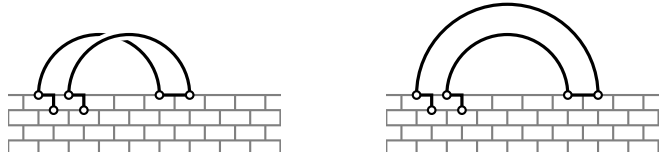

We formulate a lemma that combines connectors and linkages. 
For later application, we specify how we interpret a pathological case:
If $W$ is a wall and $\mathcal A=\emptyset$, then we accept an empty path set $\mathcal P=\emptyset$
as a $\mathcal A$--$W$-connector of any size~$s$. Indeed, the definition states 
that there should be a set $\mathcal P_A$ of $s$ disjoint paths for every $A\in\mathcal A$, 
and thus, as there is no $A$ when $\mathcal A=\emptyset$, the condition becomes vacuous. 
In the same way, the empty set $\emptyset$ is an $\emptyset$-linkage of $W$ of any size $s$.

\begin{lemma}\label{monsterlem}
For all positive integers $n,s',t'$, 
there are integers $t=t(n,s',t')$ and $f=f(n,s')$ such that the following holds.
If $\mathcal A_1,\mathcal A_2$ are sets of at most $n$ vertex sets, 
and if $W$ is a wall of size at least~$t$
with $B$ as its set of branch vertices, then either
\begin{enumerate}[label=\rm(\roman*)]
\item there is $A\in\mathcal A_1$ and a vertex set $X$ of size at most~$f$ 
such that $X$ meets every $A$--$B$-path; or
\item there is $A\in\mathcal A_2$ and a vertex set $X$ of size at most~$f$ 
such that $X$ meets every $B$--$A$--$B$-path; or
\item there is a subwall $W'$ of size at least~$t'$ and there 
is an $\mathcal A_1$--$W'$-connector $\mathcal C$ of size at least~$s'$,
and an in-series $\mathcal A_2$--linkage $\mathcal L$ of $W'$ of size at least~$s'$
such that $\mathcal C\cup\mathcal L$ is clean and all paths in $\mathcal C\cup\mathcal L$
are pairwise disjoint.
\end{enumerate}
\end{lemma}
Note that~(i) cannot happen if $\mathcal A_1=\emptyset$: Indeed,~(i) postulates the existence of some $A\in\mathcal A_1$. 
In the same way,~(ii) is impossible if $\mathcal A_2=\emptyset$.
\begin{proof}
The proof is very similar to the proof of Theorems~\ref{connlem} and~\ref{linkagelem}.
First use Lemma~\ref{connlemX} to obtain $\mathcal A_1$--$W'$-connectors for each $A_1\in\cA_1$ and $A_2$--$W'$-linkages for each $A_2\in \cA_2$. 
Via Lemma~\ref{combilem}, we turn these path systems into disjoint path systems.
Afterwards, we apply Lemma~\ref{pureLink} to find pure $A_2$--$W'$-linkages.
As indicated in Figure~\ref{seriesfig}, we turn the pure linkages into in-series linkages ($W'$ loses its top row, and we proceed with a $W'$ whose size has decreased by $1$; we prolong the other paths/linkages).
Finally, we clean the connectors and linkages with Lemma~\ref{cleanlinesslem}. 
If $\mathcal A_1=\emptyset$, we set $\mathcal C=\emptyset$, and if $\mathcal A_2=\emptyset$, we put $\mathcal L=\emptyset$. 
\end{proof}

Instead of limiting ourselves to only the extreme cases 
where we either find connectors and linkages for all of $\mathcal A_1$ and $\mathcal A_2$ or disconnect just one of the sets in either $\mathcal A_1$ or $\mathcal A_2$,
we can also observe the correctness for intermediate cases. Essentially, we can find some sets in $\mathcal A_1$ and $\mathcal A_2$ that can be disconnected while all other sets can be linked nicely.

\begin{theorem}\label{labeltheorem}
For all positive integers $n,s',t'$, 
there are integers $t=t(n,s',t')$ and $f=f(n,s')$ such that the following holds.
If $\mathcal A_1,\mathcal A_2$ are sets of at most $n$ vertex sets, 
and if $W$ is a wall of size at least~$t$, then there is a subwall $W'$ of size at least~$t'$
with $B$ as its set of branch vertices and subsets $\mathcal A_1' \subseteq \mathcal A_1$ and 
$\mathcal A_2' \subseteq \mathcal A_2$ such that 

\begin{enumerate}[label=\rm(\roman*)]
\item there is a vertex set $X$ of size at most $f$ such that $X$ meets every $A$--$B$-path for $A \in \mathcal A_1'$ 
and every $B$--$A$--$B$-path for $A \in \mathcal A_2'$, and such that $X$ is disjoint from $W'$; 
\item there is an $(\mathcal A_1 \setminus \mathcal A_1')$--$W'$-connector $\mathcal C$ of size at least~$s'$,
and an in-series $(\mathcal A_2 \setminus \mathcal A_2')$--linkage $\mathcal L$ of $W'$ of size at least~$s'$
such that $\mathcal C\cup\mathcal L$ is clean and all paths in $\mathcal C\cup\mathcal L$
are pairwise disjoint.
\end{enumerate}
\end{theorem}
Note that we view the empty set as clean, so that~(ii) also makes sense if $\mathcal C=\mathcal L=\emptyset$. 
Moreover, observe that if $\mathcal A'_1=\emptyset$ and $\mathcal A'_2=\emptyset$, then the empty set will do as $X$.

\begin{proof}
Let $$t = t(n,s', t') = t_{\ref{monsterlem}}(n,s',t')\cdot \left( 3f_{\ref{monsterlem}}(n, s') \right)^{n}$$ and 
$$f=f(n, s') = 2n\cdot f_{\ref{monsterlem}}(n, s'),$$ 
where $t_{\ref{monsterlem}}$ and $f_{\ref{monsterlem}}$ are the respective functions from Lemma~\ref{monsterlem}.

Let $G$ be a graph with a wall $W$ of size at least $t$. We start with $\mathcal A_i^0 = \emptyset$ for $i\in [2]$, $X_0=\emptyset$ and $W_0 = W$. 
We will iteratively find subwalls $W_j$ of $W$, vertex sets $X_j$ and subsets $\mathcal A_1^j\subseteq\mathcal A_1$ 
and $\mathcal A_2^j\subseteq\mathcal A_2$, and
we will write $B_j$ for the 
set of branch vertices of $W_j$.
To generate these sets, 
we iteratively apply Lemma~\ref{monsterlem} to $G-X_j$, $W_j$, $\mathcal A_1\sm\mathcal A_1^j$ and $\mathcal A_2\sm\mathcal A_2^j$ 
until outcome~(iii) of Lemma~\ref{monsterlem} occurs, or until $\mathcal A_1\sm\mathcal A_1^j=\mathcal A_2\sm\mathcal A_2^j=\emptyset$.
In this process, we will guarantee that, in the $j$-th iteration, the following conditions are satisfied:
\begin{enumerate}[label=\rm(\alph*)]
    \item the set $X_j$ has size $|X_j|\leq j\cdot f_{\ref{monsterlem}}(n, s')$ and meets every $A$--$B_j$-path for $A \in \mathcal A_1^j$ 
and every $B_j$--$A$--$B_j$-path for $A \in \mathcal A_2^j$; \label{iti}
	\item the subwall $W_j$ of $W$ has size at least $t_{\ref{monsterlem}}(n,t',s')\cdot \left(3f_{\ref{monsterlem}}(n, s')\right)^{n-j/2}$
	and is disjoint from $X_j$; and\label{itii}
	\item $|\mathcal A_1^{j}|+|\mathcal A_2^{j}|= j$.\label{itiii} 
\end{enumerate}

As the process ends at the latest when $\mathcal A_1\sm\mathcal A_1^j$ and $\mathcal A_2\sm\mathcal A_2^j$ become empty, 
the process will end after at most $2n$ iterations. 

For $j=0$, the above conditions are true, so let $j\geq 1$. 
Applying Lemma~\ref{monsterlem} to $G-X_{j-1}$, $W_{j-1}$, 
 $\mathcal A_1\sm\mathcal A_1^{j-1}$, $\mathcal A_2\sm\mathcal A_2^{j-1}$, $n,s'$ and $t'$ 
yields one of the three outcomes~(i), (ii), and~(iii) of the lemma.
Note that the size of $W_{j-1}$ is large enough in any iteration.

First, assume that outcome~(i) or~(ii) of Lemma~\ref{monsterlem} occurs. 
That is, there is a set $X'$ of size $|X'|\leq f_{\ref{monsterlem}}(n, s')$ 
and an element
\[
A_j \in (\mathcal A_1\sm \mathcal A_1^{j-1})\cup (\mathcal A_2\sm\mathcal A_2^{j-1})
\]
such that $X'$ meets either every $A_j$--$B_{j-1}$-path or every $B_{j-1}$--$A_{j}$--$B_{j-1}$-path
in the graph $G-X_{j-1}$.

We put $X_j=X_{j-1}\cup X'$ and set $\mathcal A_1^j=\mathcal A_1^{j-1}\cup\{A_j\}$ and $\mathcal A_2^j=\mathcal A_2^{j-1}$,
or $\mathcal A_2^j=\mathcal A_2^{j-1}\cup\{A_j\}$ and $\mathcal A_1^j=\mathcal A_1^{j-1}$, depending on whether $A_j$ came 
from $\mathcal A_1$ or from $\mathcal A_2$.
This already ensures~\ref{itiii},
and~\ref{iti} with $B_{j-1}$ in the role of $B_j$. Since we will find $W_j$ as a subwall of $W_{j-1}$, 
we have $B_{j}\subseteq B_{j-1}$, and~\ref{iti} will also be true for $B_j$. 

To find the subwall $W_j$, we split $W_{j-1}$ in a grid-like fashion 
into $2 f_{\ref{monsterlem}}(n, s')$ disjoint
subwalls of size at least $t_{\ref{monsterlem}}(n,t',s')\cdot (3f_{\ref{monsterlem}}(n, s'))^{n-j/2}$.
This is possible as $W_{j-1}$ has size at least $t_{\ref{monsterlem}}(n,t',s')\cdot (3f_{\ref{monsterlem}}(n, s'))^{n-(j-1)/2}$,
by~\ref{itii}. Since $|X'|\leq f_{\ref{monsterlem}}(n, s')$, one of these subwalls is disjoint from
$X'$; pick one to be $W_j$. Note that $W_j$ is also disjoint from $X_j=X_{j-1}\cup X'$, as $W_{j-1}$
is disjoint from $X_{j-1}$ by~\ref{itii}. This~proves~\ref{itii} for $W_j$.

If, in every iteration, outcome~(ii) or~(iii) occurs, and the process ends with 
$\mathcal A_1\sm\mathcal A_1^{2n}=\emptyset$ and $\mathcal A_2\sm\mathcal A_2^{2n}=\emptyset$, 
then we set $\mathcal A'_1=\mathcal A_1^{2n}$, $\mathcal A'_2=\mathcal A_2^{2n}$, $X=X_{2n}$, $W'=W_{2n}$
and $\mathcal C=\mathcal L=\emptyset$. By~\ref{iti} and~\ref{itii}, the conclusions 
of the theorem hold true.

\medskip 
What happens if the application of Lemma~\ref{monsterlem} yields outcome~(iii) in iteration~$j$?
Then, we find a subwall $W'$ of $W_{j-1}$ of size at least $t'$,  
an $(\mathcal A_1\sm\mathcal A_1^{j-1})$--$W'$-connector $\mathcal C$ of size at least~$s'$,
and an in-series $(\mathcal A_2\sm\mathcal A_2^{j-1})$--linkage $\mathcal L$ of $W'$ of size at least~$s'$
such that $\mathcal C\cup\mathcal L$ is clean and all paths in $\mathcal C\cup\mathcal L$
are pairwise disjoint. 
We define $X=X_{j-1}$, $\mathcal A_1' = \mathcal A_1^{j-1}$ and $\mathcal A_2' = \mathcal A_2^{j-1}$. 
Again, we are done with~\ref{iti} and~\ref{itii}.
\end{proof}

\bibliographystyle{amsplain}
\bibliography{erdosposa}

\vfill

\small
\vskip2mm plus 1fill
\noindent
Version \today{}
\bigbreak

\noindent
Henning Bruhn {\tt <henning.bruhn@uni-ulm.de>}\\
Arthur Ulmer\\
Institut f\"ur Optimierung und Operations Research\\
Universit\"at Ulm, Ulm\\
Germany\\

\bigskip

\noindent
Felix Joos
{\tt <joos@informatik.uni-heidelberg.de>}\\
Institut f\"ur Informatik\\
Universit\"at Heidelberg, Heidelberg\\
Germany\\

\end{document}